\documentclass[11pt]{article}

\usepackage{amsthm,amsmath,amssymb}

\usepackage{graphicx}

\usepackage[colorlinks=true,citecolor=black,linkcolor=black,urlcolor=blue]{hyperref}

\usepackage{stmaryrd}
\usepackage{textcomp}
\usepackage{enumerate}
\usepackage{setspace}
\usepackage{amssymb}
\usepackage{amsthm}
\usepackage{amsmath}
\usepackage{graphicx}
\usepackage{extarrows}
\usepackage{marvosym}
\usepackage{empheq}
\usepackage{latexsym}
\usepackage{endnotes}
\usepackage{fontenc}

\begin{document}
\author{Julian Sahasrabudhe}
\title{Monochromatic Solutions to Systems of Exponential Equations}
\maketitle

\newtheorem{theorem}{Theorem}
\newtheorem{definition}[theorem]{Definition}
\newtheorem{lemma}[theorem]{Lemma}
\newtheorem{claim}[theorem]{Claim}
\newtheorem{subclaim}[theorem]{Sub-Claim}
\newtheorem{corollary}[theorem]{Corollary}
\newtheorem{problem}[theorem]{Problem}
\newtheorem{fact}[theorem]{Fact}
\newtheorem{observation}[theorem]{Observation}
\newtheorem{conjecture}[theorem]{Conjecture}
\newtheorem{question}[theorem]{Question}
\newcommand{\Addresses}{{
\bigskip
\footnotesize
\textsc{ Julian Sahasrabudhe, Department of Mathematics, University of Memphis, Memphis Tennessee, USA}\par\nopagebreak
 \texttt{julian.sahasra@gmail.com}
}}

\begin{abstract}
Let $n\in \mathbb{N}$, $R$ be a binary relation on $[n]$, and $C_1(i,j),\ldots,C_n(i,j) \in \mathbb{Z}$, for $i,j \in [n]$. We define the exponential system of equations $\mathcal{E}(R,(C_k(i,j)_{i,j,k})$ to be the system
\[ X_i^{Y_1^{C_1(i,j)} \cdots Y_n^{C_n(i,j)}  } =  X_j , \text{  for } (i,j) \in R ,
\] in variables $X_1,\ldots,X_n,Y_1,\ldots,Y_n$. The aim of this paper is to classify precisely which of these systems admit a monochromatic solution ($X_i,Y_i \not=1)$ in an arbitrary finite colouring of the natural numbers. This result could be viewed as an analogue of Rado's theorem for exponential patterns.
\end{abstract}

\section{Introduction}
In 2011, Sisto \cite{Si} made the surprising observation that an arbitrary 2-colouring of the natural numbers admits infinitely many integers $a,b>1$ such that $a,b,a^b$ all receive the same colour. He went on to ask if a similar result holds for colourings of the natural numbers with more colours. Brown \cite{Br}, simplifying and extending the proof of Sisto, gave further examples of exponential, monochromatic patterns that are present in an arbitrary $2$-colouring and also proved some weaker results for monochromatic patterns in more colours. In \cite{ExpPatterns} we answered Sisto's question by showing that any \emph{finite} colouring of the positive integers admits $a,b >1$ such that $a,b,a^b$ are monochromatic and went on to develop, in this context, a theory of patterns defined by compositions of the exponential function. In the present paper we turn from the study of patterns arising as compositions of the exponential function, to understand exponential patterns that arise as solutions to systems of equations. 
\paragraph{} 
The motivation for the study of monochromatic solutions to equations lies in the seminal work of Rado \cite{Ra}, who classified the systems of homogeneous linear equations that admit a solution in an arbitrary finite colouring of the natural numbers. More precisely, we say that an $m \times n$ matrix $A$ is \emph{partition regular} if every finite colouring of $\mathbb{N}$ admits monochromatic $x_1,\ldots,x_n \in \mathbb{N}$, for which $Ax = 0$, where $x = (x_1,\ldots,x_n)$. Rado classified the partition regular matrices by giving a simple criterion on the columns of such matrices (to be recalled in section \ref{sec:Preliminaries}). It is in this spirit that the present paper sets out.
\paragraph{}
It is worth pointing out that, even in the classical, linear theory, there is a distinction between studying patters which solve linear systems, and patterns which arise as fixed linear compositions of several free variables. These two types of partition regularity are sometimes termed ``kernel partition regular'' and ``image partition regular'', respectively. So, while Rado's theorem gave a complete understanding of what systems can be solved in an arbitrary colouring, it was not until the work Hindman and Leader \cite{ImagePartReg} that a classification of image partition regular systems was fully understood. We refer the reader to the survey of Hindman \cite{HindSurvey}, for details.
\paragraph{}
Before going further, let us recall some terminology. Let $k \in \mathbb{N}$ and $X$ be a non-empty set. We call a function $f  : \mathbb{N} \rightarrow X$ a \emph{finite colouring} if $X$ is finite, and a $k$-\emph{colouring}, if $|X| \leq k$. We refer to the elements of $X$ as \emph{colours}. We say that a collection $\mathcal{A}$, of ordered tuples of integers, is \emph{partition regular} if for every finite colouring $f: \mathbb{N} \rightarrow X$ we can find $n \in \mathbb{N}$ and $x_1,\ldots,x_n \in \mathbb{N}$, such that $f(x_1) = \cdots = f(x_n)$ and $(x_1,\ldots,x_n) \in \mathcal{A}$. It shall also be  convenient to use the notation $\star$ to denote exponentiation. That is, define the binary operation $\star$ as $a \star b  = a^b$, for $a,b \in \mathbb{N}$. 
\paragraph{}
For $n \in \mathbb{N}$, let $R$ be a binary relation on $[n]$. Given integers $C_1(i,j),\ldots,C_n(i,j) \in \mathbb{Z}$, for $i,j \in [n]$, we define the system of equations $\mathcal{E}\left(R,\{C_k(i,j)\}_{i,j,k}\right)$ by 
\begin{equation} X_i^{Y_1^{C_1(i,j)} \cdots Y_n^{C_n(i,j)}  }= X_j , \text{  for } (i,j) \in R,
\end{equation} where $X_1,\ldots,X_n,Y_1,\ldots,Y_n$ are variables.
\paragraph{}
Our main result will tell us that the above system of equations has a monochromatic solution in every finite colouring if and only if 
an associated system of \emph{linear} equations in the variables $Y_1,\ldots,Y_n$ has a solution. Of course, this will result in a classification, by appealing to Rado's Theorem, mentioned above.
\paragraph{}
To define the \emph{associated linear system}, $\mathcal{L}\left(R,\{C_k(i,j)\}_{i,j,k}\right)$, we treat $R$ as a directed graph $D = ([n],R)$ and let $\mathcal{L}\left(R,\{C_k(i,j)\}_{i,j,k}\right)$ be the system of equations, indexed by the (not necessarily directed) cycles $C$ of $D$,
\begin{equation} \label{equ:AssLinSystem} \sum_{e \in C} (-1)^{d(e)}\left(C_1(e)Y_1 + \ldots + C_n(e)Y_n \right) = 0,
 \end{equation} where, for each cycle $C$, we fix some orientation and then define $d(e) = 0$ if the edge $e$ is oriented in the same way as the cycle and $d(e) = 1$ if the orientation of the edge and the cycle are different. 
We may now state our main theorem. 
\begin{theorem} \label{thm:MainClassification}
For $n \in \mathbb{N}$, let $R \subseteq [n]\times [n]$, and $C_k((i,j)) \in \mathbb{Z}$, for each $i,j,k \in [n]$. The system of exponential equations $\mathcal{E}\left(R,\{C_k(i,j)\} \right)$ is partition regular if and only if $\mathcal{L}\left(R,\{C_k(i,j)\} \right)$ is partition regular. \end{theorem}

So, for example, every finite colouring of the positive integers admits a monochromatic solution to the equation  
\[ X_1^{Y_1 \cdot Y_2} = X_2^{Y_3\cdot Y_4} ,
\] with $X_1,X_2,Y_1,Y_2,Y_3,Y_4 >1 $. While there is a colouring forbidding monochromatic solutions to equations of the form
\[ X^{Y^2} = X^{Z},
\] with $X,Y,Z>1$. 

\paragraph{}
One can readily get a feel for the``only if'' implication in the theorem. The first step is to notice that we may obtain a system of the form $\mathcal{L}$ from the associated exponential system $\mathcal{E}$, by first applying $\nu \circ \nu$ to both sides of the equations in $\mathcal{E}$, where $\nu$ is a logarithm-type function. That is, $\nu$ is an appropriate function which satisfies $\nu(a^b) = b\nu(a)$. If we then take linear combinations to eliminate the terms of the form $\nu^2(X_i)$, we obtain a system of the general shape of $\mathcal{L}$. From here it is not hard to see that if $c$ is a colouring forbidding a monochromatic solution to $\mathcal{L}$, then $c(\nu(x))$ is a colouring forbidding a solution to the original, exponential equation. 


\paragraph{}
To prove the ``if'' direction of the theorem, we show that if $\mathcal{A} \subseteq \mathbb{N}^n$ is partition regular we can ``lift'' $\mathcal{A}$ to find an associated exponential pattern that is also partition regular. More precisely, if $n \in \mathbb{N}$, $\mathcal{A} \subset \mathbb{N}^n$, and $W : \mathbb{N}^n \rightarrow \mathbb{N}$ is an arbitrary ``weight'' function, we define the \emph{exponential} $\mathcal{A}$\emph{-patten with weight} $W$, as follows. For each $(x_1,\ldots,x_n) \in \mathcal{A}$, we include into the associated exponential pattern, the following $(W(x_1,\ldots,x_n) + n +1)$-tuple consisting of the elements 
\[ a ,b^{x_1}, \ldots,b^{x_n},  
\] along with
\[ a^{b^1},a^{b^2}, \ldots , a^{b^{W(x_1,\ldots, x_n)}},
\] for each $a,b>1$. We show the following.

\begin{theorem} \label{thm:LiftingPatterns}
Let $n \in \mathbb{N}$, and $W : \mathbb{N}^n \rightarrow \mathbb{N}$ be a function. If $\mathcal{A} \subseteq \mathbb{N}^n$ is partition regular, then the associated exponential $\mathcal{A}$-pattern, with weight $W$, is also partition regular.
\end{theorem}
\paragraph{}
In the next section, we quickly recall some of the theory relevant to this paper and introduce the central definitions. 
In Section~\ref{sec:ProofOfLiftingThm}, we prove Theorem~\ref{thm:LiftingPatterns}. Finally, in Section~\ref{sec:ProofOfClassification}, we use this result to deduce our classification theorem, Theorem~\ref{thm:MainClassification}.

\section{Preliminaries} \label{sec:Preliminaries}

\paragraph{}
We start by recalling Rado's classical theorem on partition regular systems of linear equations \cite{Ra}. For $m,n\in \mathbb{N}$, we let $A$ be a $m\times n$ matrix with integer entries. We say that $A$ is \emph{partition regular} if the collection $\{ x \in \mathbb{N}^n : Ax = 0 \}$ is partition regular. If we let $v_1,\ldots,v_n \in \mathbb{Z}^n$ denote the column vectors of $A$, we say that $A$ satisfies the \emph{columns property} if one can partition $[n] = S_0 \cup \cdots \cup S_d$, for some $d \in [0,n-1]$, so that $ \sum_{j \in S_0}  v_j = 0$, while  $\sum_{j \in S_i} v_j $, lies in the $\mathbb{Q}$-linear span of the vectors of $S_0 \cup \cdots \cup S_{i-1}$, for each $i \in [n]$. Rado's theorem establishes that these two properties of $A$ are equivalent. 
\begin{theorem} \label{thm:RadoTheorem}
For $m,n \in \mathbb{N}$, let $A$ be an $m\times n$ matrix with integer entries, then $A$ is partition regular if and only if $A$ has the columns property. 
\end{theorem}
Although we shall not require them explicitly in the present paper, it is convenient to recall the \emph{Rado colourings}. These colourings were introduced by Rado to demonstrate the non-partition regularity of matrices without the columns property \cite{Ra} (See \cite{GRS}). For a prime $p$ and $x \in \mathbb{N}$, we define the colouring $c_p : \mathbb{N} \rightarrow [p-1]$ by defining $c_p(x)$ as the cofficient of $p^k$ in the base-$p$ expansion of $x$, where $k$ is the largest integer so that $p^k$ divides $x$. Rado proved that if $A$ fails to have the columns property then, for sufficiently large primes $p$, if $x_1,\ldots,x_n$ are integers such that $c_p(x_1) =  \cdots = c_p(x_n)$ then $Ax \not= 0$, where $x = (x_1,\ldots,x_n)$.
\paragraph{}
Turning now to introduce some important notions, let $f : \mathbb{N} \rightarrow [k]$ be a $k$-colouring and let $f_{n}: \mathbb{N} \rightarrow [k] $ be $k$-colourings, for $n \in \mathbb{N}$. We say that the sequence $\{ f_n \}$ \emph{converges to} $f$ and write $f_n\rightarrow f$ if for every $m \in \mathbb{N}$ there exists some $M \in \mathbb{N}$ so that for all $n \geq M$, $f_n(x) = f(x)$, for all $x \in [m]$. As expected, we also say that a sequence \emph{converges} if there exists some $f$ that the sequence converges to. The following basic fact on sequences of colourings $f : \mathbb{N} \rightarrow [k]$ is often referred to as the \emph{compactness property}.

\begin{fact} \label{Compactness} Given a sequence of colourings $f_{n} : \mathbb{N} \rightarrow [k]$ there exists some 
$f : \mathbb{N} \rightarrow [k] $ and a strictly increasing sequence $\{N(n) \}_n \subseteq \mathbb{N}$ for which
\[ f_{N(n)} \rightarrow f, \] as $ n \rightarrow \infty$.
\end{fact} \qed

We also make use of the following consequence of the compactness property. Given a partition regular collection $\mathcal{A} \subseteq \mathbb{N}^n$ and a positive integer $k$, there exists a minimum integer $P(\mathcal{A};k)$ such that every $k$-colouring of $\mathbb{N}$ admits a monochromatic $(x_1,\ldots,x_n) \in \mathcal{A}$ with $x_1,\ldots,x_n \leq P(\mathcal{A};k)$. 
\paragraph{}
We make considerable use of van der Waerden's classical theorem \cite{vdW}, which states that for every $k,l \in \mathbb{N}$, there exists a minimal integer $W_k(l)$ such that every $k$-colouring of an arithmetic progression of length $W_k(l)$ contains a monochromatic sub-progression of length $l$.
\paragraph{}
Now, for $r \in \mathbb{N}$, let $f_1, \ldots ,f_r : \mathbb{N} \rightarrow [k]$ be $k$-colourings. We call a sequence of colours
$c_1, \ldots ,c_r \in [k]$ \emph{large for} $f_1, \ldots ,f_r$ if for every $M \in \mathbb{N}$ we can find a progression $P_M$ of length $M$, such that $f_i(P_M) = c_i$, for each $i \in [r]$. The following two facts now follow easily from van der Waerden's theorem. 

\begin{corollary} \label{cor:vdw}
If $f$ is a finite colouring of $\mathbb{N}$, there exists a colour that is large with respect to $f$.  \qed
\end{corollary}

\begin{lemma} \label{ExtentionLemma} For $r \in \mathbb{N} $, let $f_1, \ldots ,f_{r},f_{r+1}, : \mathbb{N} \rightarrow [ k ]$ be $k$-colourings. If $c_1, \ldots ,c_r \in [ k ]$ is large with respect to $f_1, \ldots ,f_r$, then there exists a colour $c_{r+1} \in [k]$ so that 
$c_1,\ldots,c_{r+1}$ is large with respect to $f_1,\ldots,f_{r+1}$.
\end{lemma}
\begin{proof}
Let $M \in \mathbb{N}$ be a parameter. Now use the fact that $c_1, \ldots , c_r$ is large for $f_1, \ldots , f_r$ to find a progression $P_M$ of length $W_k(M)$ so that $f_i(P_M) = \{ c_i \}$ for $ i \in [r]$. Applying van der Waerdens's theorem to the colouring $f_{r+1}$, along the progression $P_M$, we obtain a monochromatic progression $P'_M \subseteq P_M$ of length $M$, for which $f_i(P'_M) = c_i$, for $i\in [r]$ and $f_{r+1}(P'_M) = c(M)$, for some $c(M) \in [k]$. To finish, apply the above for all choices of $M \in \mathbb{N}$. There is some value of $[k]$ that is attained infinitely often as a value of $c(M)$. We set $c_{r+1}$ to be this value.
\end{proof}
For a each positive integer $d$, we define the $d$-sequence to be the sequence of integers $\{ 2^{d2^x} \}_x$. Given a colouring $f : \mathbb{N} \rightarrow [k]$, we record the colouring ``restricted'' to the $d$-sequence as the colouring $f_{d} : \mathbb{N} \rightarrow [k]$, defined by $f_d(x) = f\left(2^{d2^{x}} \right)$, for $x \in \mathbb{N}$.

\section{Lifting partition regular patterns} \label{sec:ProofOfLiftingThm}
 
In this section we prove Theorem~\ref{thm:LiftingPatterns} on ``lifting'' partition regular patterns. The core of the proof is contained in the following lemma, which works by constructing, at each stage, a huge number of sequences which approximate some ``idealized'' colourings $\widetilde{f}_1,\ldots,\widetilde{f}_r$. Each ``idealized'' colouring will then act as a mold to help us look for future sequences which, in turn, approximate a new idealized colouring $\tilde{f}_{r+1}$. Assuming that we don't find the appropriate pattern, we shall observe that the colouring along our new sequences becomes more and more restricted, until we obtain a contradiction. 
\paragraph{}
We should note that it is possible to replace our infinitary arguments with finitary ones. However, this exchange would come at the cost of added clutter and difficulty for the reader. 

\begin{lemma} \label{MainLemma}
For $n \in \mathbb{N}$, let $W : \mathbb{N}^n \rightarrow \mathbb{N}$ be a function and let $\mathcal{A} \subseteq \mathbb{N}^n$ be partition regular. If $r,k \in \mathbb{N}$ and $f : \mathbb{N} \rightarrow [k] $ is a colouring that admits no monochromatic, exponential $\mathcal{A}$-system with weight $W$, then we can find colours $c_1, \ldots ,c_{r}$
and corresponding colourings $\widetilde{f}_1, \ldots , \widetilde{f}_r : \mathbb{N} \rightarrow [k]$ so that the following hold.
\begin{enumerate}
\item  For each $i \in [r]$ there exists a sequence of integers $\{d_i(n)\}_n$ so that

\[  f_{d_i(n)} \rightarrow \widetilde{f}_i \ \ \text{as} \ \ n \rightarrow \infty, 
\] where $f_{d_i(n)}$ is the colouring $f$ restricted to the $d_i(n)$-sequence;

\item The the sequence of colours $c_1,\ldots,c_r$ is large with respect to $\widetilde{f}_1,\ldots, \widetilde{f}_r$;

\item $c_1,\ldots,c_r$ are distinct.
\end{enumerate}
\end{lemma}

\begin{proof}
We apply induction on $r$. For $r = 1$, choose $d_1(n) = 1$ for all $n \in \mathbb{N}$. Thus $f_{d_1(n)}$ trivially converges to a $k$-colouring $\tilde{f}_1$. Now, by the corollary to van der Waerden's theorem above (Corollary \ref{cor:vdw}), there exists $c_1 \in [k]$ that is large with respect to $\widetilde{f}_1$. This proves the base case of the induction.
\paragraph{}
For the inductive step, suppose that we have found colourings $\widetilde{f}_1, \ldots ,\widetilde{f}_{r-1}$ with associated colours $c_1, \ldots ,c_{r-1}$ that satisfy the statement of the lemma. In what follows, we let $M \in \mathbb{N}$ be a parameter. Now since $c_1, \ldots ,c_{r-1}$ is large for $\widetilde{f}_1, \ldots , \widetilde{f}_{r-1}$, we may find a progression $P_M = \{ d(M)x + a(M) : x \in [M] \}$ of length $M$ with $\tilde{f}_i (P_M) = \left\lbrace c_i \right\rbrace$, for each $i \in [r-1]$. Now let $M' \in \mathbb{N}$ be a (new) parameter and define 
\[ h(M') = \max\left\lbrace W(x_1,\ldots,x_n) : x_1,\ldots,x_n \in [P(\mathcal{A};k^{M'})] \right\rbrace ,\] while recalling that $P(\mathcal{A};k)$ is the smallest integer so that every $k$-colouring of $\mathbb{N}$ admits a monochromatic $x_1,\ldots,x_n$ such that $(x_1,\ldots,x_n) \in \mathcal{A}$ and $x_1,\ldots,x_n \leq P(\mathcal{A};k)$. For each appropriate $M,M'$, we consider numbers of the form $ 2^{d(M)x2^{y}} $, where $x \leq P(\mathcal{A};k^{M'})$ and $y \leq M'$. In particular, we define a colouring $F = F_{M,M'} :[P(\mathcal{A}; k^{M'})] \rightarrow [k^{M'}]$ by
\[ F(x) = \left( f\left(2^{d(M)x2^1}\right) , f\left(2^{d(M)x2^2}\right), \ldots , f\left(2^{d(M)x2^{M'}}\right) \right) 
\] and observe that $F$ defines a $k^{M'}$-colouring of $[P(\mathcal{A}; k^{M'})]$ and therefore we can find a set of positive integers $x_1 = x_1(M,M'),\ldots,x_n = x_n(M,M')$, with $(x_1,\ldots,x_n) \in \mathcal{A}$, which is monochromatic with respect to the colouring $F$. 
We now show that the colouring along the sequence 
\[ 2^{d(M)x_12^1},\ldots, 2^{d(M)x_12^{M'}}
\] is rather constrained, provided $M$ is sufficiently large compared to $M'$.

\begin{claim} \label{atmostlElementsInSequence} If $M \geq  2^{M'}h(M')$, then none of the elements 
\[ 2^{d(M)x_12^1} , \ldots , 2^{d(M)x_12^{M'}} 
\] receive any of the colours $c_1,\ldots,c_{r-1}$.
\end{claim}
\begin{proof}
We show that if just one of these elements is coloured by a colour of $\{c_1,\ldots,c_p \}$, we can find a monochromatic exponential $\mathcal{A}$-system, thus obtaining a contradiction. So, assume that there is an element $ 2^{d(M)x_12^y} $, with $y \in [M']$, that receives colour $c_p$, with $p \in [r-1]$. By the definition of the $x_1,\ldots,x_n$, it follows that all of the elements 
\[ 2^{d(M)x_12^y} ,2^{d(M)x_22^y}, \ldots, 2^{d(M)x_n 2^y} 
\] receive colour $c_p$. 
\paragraph{}
Next, choose an integer $N = N(M)$ to be large enough so that the colouring $f_{d_p(N)}$ agrees with $\widetilde{f}_{p}$ in at least $\max P(M)$ places. Such a choice of $N$ exists, as we are granted $f_{d_p(n)} \rightarrow \widetilde{f}_p$ as $n \rightarrow \infty$, by the induction hypothesis. As a result, we have that $P_M$ is coloured by $f_{d_p(N)}$ exactly as it is coloured by $\widetilde{f}_p$. 
\paragraph{}
We claim that $(x_1,\ldots,x_n)$, $a = 2^{d_p(N)2^{a(M)}}$, and $b = 2^{d(M)2^y}$ define an exponential $\mathcal{A}$-system that is monochromatic in the colour $c_p$. We already know that $f(b^{x_1}) = \cdots =f(b^{x_n}) = c_p$, so it only remains to check the colour of $a \star (b^l)$, for each $l \in [W(x_1,\ldots,x_n)]$. So fix $l \in [W(x_1,\ldots,x_n)]$ and write
\[ f\left(a^{b^l} \right) = f\left( 2 \star \left( d_p(N)2^{ a(M) + d(M)l2^y } \right)\right) = f_{d_p(N)}\left( a(M) + d(M)l2^y \right).
\] Now since $l2^y \leq W(x_1,\ldots,x_n)2^{M'} \leq M$, our choice of $N$ allows us to conclude that the above is equal to 
\[ \tilde{f}_p\left( a(M) + d(M)l2^y \right) = c_p ,
\] where this last inequality holds as $P_M = \{a(M) + d(M)x : x \in [M] \}$ is a progression with the property that $\tilde{f}_p(P_M) = c_p$, as we assumed above. 
\paragraph{}
Hence we have found a exponential $\mathcal{A}$-system, monochromatic in $c_p$. This contradicts the assumption on $f$ and completes the proof of the claim. 
\end{proof}
So for each $M' \in \mathbb{N}$, we set $d'(M') = d\left(2^{M'}h(M')\right)x_1\left(2^{M'}h(M'),M'\right)$ and apply the compactness property (i.e. Fact~\ref{Compactness}) to find a subsequence of the $\{d'(M')\}$ for which the sequence of colourings $\left( f_{d'(M')} \right)_{M' \in \mathbb{N}}$ converges. We take $\{ d_r(M)\}_M$ to be this subsequence and $\tilde{f}_r$ to be the corresponding limiting colouring. 
\paragraph{}
Now note that we have $c_1,\ldots,c_{r-1} \not\in \tilde{f}_r(\mathbb{N})$, for if $\tilde{f}_r(x_0) = c_i$ for some $x_0 \in \mathbb{N}$ and $i \in [r-1]$, it would follow, for sufficiently large $M'$, that the integer 
\[ 2 \star \left( d'(M')2^{x_0} \right) = 2 \star \left( d\left(2^{M'}h(M')\right)x_1\left(2^{M'}h(M'),M'\right)2^{x_0} \right), 
\] would receive the colour $c_i$, which is in contradiction with Claim~\ref{atmostlElementsInSequence}.
\paragraph{} 
Finally, we choose the colour $c_r$. This is easily done; since $c_1, \ldots ,c_{r-1}$ is large for $\widetilde{f}_1, \ldots , \widetilde{f}_{r-1}$, by Lemma \ref{ExtentionLemma}, we may find a colour $c_r \in [k]$ so that $c_1, \ldots ,c_r$ is large with respect to $\tilde{f}_1, \ldots , \tilde{f}_r$. It is clear that the colour $c_r$ must be distinct from $c_1,\ldots,c_{r-1}$ as $c_1,\ldots,c_{r-1} \not\in \tilde{f}_r\left( \mathbb{N} \right)$. This concludes the induction step of the proof and hence we are done, by induction. 
\end{proof}

\section{Proof of Theorem~\ref{thm:MainClassification}} \label{sec:ProofOfClassification}

We are now in a position to prove our classification of partition regular, exponential systems. Recall that a binary relation $R$ comes implicitly in the definition of the systems $\mathcal{L}$ and $\mathcal{R}$. In what follows, we regard this relation as a directed graph in the obvious way, thus allowing us to borrow from the terminology of directed graphs. Indeed, call a digraph \emph{weakly connected} if the underlying, undirected graph is connected, and say that a subgraph of $G$ is a \emph{weak component} of $G$ if this subgraph is a component in the underlying, undirected graph. 
\paragraph{}
\emph{Proof of Theorem~\ref{thm:MainClassification} : } For $n \in \mathbb{N}$, $R \subseteq [n]^2$, and $C_k(i,j) \in \mathbb{Z}$,  $i,j,k \in [n]$, we consider the system of exponential equations $\mathcal{E} = \mathcal{E}\left(R,\{C_k(i,j)\}\right)$, along with the associated linear system of equations $\mathcal{L} = \mathcal{L}\left(R,\{C_k(i,j)\}\right)$. On a technical note, we may assume that we have no equations of the form $X_i = X_j$, $i,j \in [n]$ in our system. For if the equation $X_i = X_j$ occurs in our system, we may obtain an equivalent system by simply replacing all occurrences of $X_i$ with $X_j$ in $\mathcal{E}$ and then removing the equation $X_i = X_j$.

\paragraph{} Let us first assume that $\mathcal{L}$ is partition regular. We let $\mathcal{A} \subseteq \mathbb{N}^n$ be the collection of positive-integer solutions to the system of equations $\mathcal{L}$ and we define the weight function $W: \mathbb{N}^n \rightarrow \mathbb{N}$, by $W(x_1,\ldots,x_n) = \left(\sum_{i,j,k}|C_k(i,j)|\right)\sum_i^n x_i $, for all $x_1,\ldots,x_n \in \mathbb{N}$.
\paragraph{}
Now, given a finite colouring $f$ of the integers, apply Theorem \ref{thm:LiftingPatterns} to find integers $a,b > 1$ and $z_1,\ldots,z_n \in \mathbb{N}$ so that all of the integers 
\[ a,b^{z_1}, \ldots, b^{z_n}, a^{b^1}, \ldots,a^{b^{W(z_1,\ldots,z_n)}} 
\] are given the same colour by $f$. 
\paragraph{}
We now define numbers $x_1,\ldots,x_n,y_1,\ldots,y_n$ that will constitute a monochromatic solution to $\mathcal{E}$. We start by selecting $y_i = b^{z_i}$, for $i \in [n]$. Now assume that (wlog) that the underlying, undirected graph of $R$ has $m \in [n]$ weak components and that $1,\ldots, m$ are representative vertices from each of these $m$ components. We define the (auxiliary) integers $k_i$, $i \in [n]$, by first choosing a (not necessarily directed) path $P$ between $i$ and a ``representative'' vertex $j \in [m]$. We then set
\[ k_i = \sum_{e \in E(P)} (-1)^{d(e)}\left(C_1(e)z_1 + \cdots + C_n(e)z_n \right) .
\] Of course, the value of $k_i$ is independent of the path $P$, for if we have another path $P'$ from $i$ to $j$ we set 
\[ k'_i = \sum_{e \in E(P')} (-1)^{d(e)}\left(C_1(e)z_1 + \cdots + C_n(e)z_n \right), \] and observe that 
\[ k_i - k'_j = \sum_{e \in E(C)} (-1)^{d(e)}\left(C_1(e)z_1 + \cdots + C_n(e)z_n \right) \]
where, $C$ is the closed walk formed by first traversing $P$ and then traversing $P'$ backwards. We may then partition the edges of the closed walk $C$ as a union of cycles $C^1,\ldots,C^l$, $l \in \mathbb{N}$, and thus decompose the above sum as
\[ = \sum_{i=1}^l \sum_{e \in E(C_i)} (-1)^{d(e)}\left(C_1(e)z_1 + \cdots + C_n(e)z_n \right) ,
\] which is clearly $0$, for $z_1,\ldots,z_n$ is a solution to $\mathcal{L}$.
\paragraph{}
We now select $x_1 = \cdots = x_m = a$ and then define $x_i = a^{b^{k_i}}$, for $i > m$. We clearly have that $k_i \leq W(x_1,\ldots,x_n)$, for each $i \in [n]$ and therefore all of our choices of the $x_1,\ldots,x_n,y_1,\ldots,y_n$ receive the same colour from $f$. It only remains to check that they satisfy the system of equations $\mathcal{E}$. To this end, we note that for an edge $(i,j) \in R$, we have
\[ k_j - k_i = C_1((i,j))z_1 + \cdots + C_n((i,j))z_n ,
\] which follows from the ``independence of path'' argument above. So, finally, if $e = (i,j) \in R$ we have 
\[  x_i \star \left(y_1^{C_1(e)} \cdots y_n^{C_n(e)}\right) = a \star \left( b \star \left( k_i + C_1(e)z_1 + \cdots + C_n(e)z_n \right) \right) \]
\[ = a \star \left( b \star k_j  \right) = x_j ,
\] as desired. This proves that the $x_1,\ldots,x_n,y_1,\ldots,y_n$ indeed form a solution to $\mathcal{E}$ and thus we have shown that $\mathcal{E}$ is partition regular.  
\paragraph{}
We now show that if $\mathcal{L}$ is \emph{not} partition regular, we may produce a colouring demonstrating that $\mathcal{E}$ is not partition regular. For $x \in \mathbb{N}$, write $x$ in its prime expansion $x = p_1^{e_1} \cdots p_k^{e_k}$ and define the function $\nu(x) = e_1 + \cdots + e_k$. This function has three simple properties that will be useful for us.
\begin{enumerate}
\item If $x > 1 $ then $\nu(x) > 0$; \label{property1ofNu}
\item $\nu(xy) = \nu(x) + \nu(y)$ and, in particular, $\nu(a^b) = b\nu(a) $; \label{property2ofNu}
\item $\nu(x)$ takes integers values.
\end{enumerate}
Now assume that $\mathcal{L}$ is not partition regular and let $c$ be a finite colouring that forbids monochromatic solutions to $\mathcal{L}$. We define the colouring $f : \mathbb{N}\setminus \{1 \} \rightarrow [p-1]$ to be 
\[ f(x) = c(\nu(x)).
\] To see that this colouring has the required property, we would like to apply the function $\nu \circ \nu$ to both sides of each of our equations. We encounter a tiny wrinkle as $\nu(x)$ is only defined when $x > 1$ and thus $\nu(\nu(x))$ is only defined when $x$ is a composite integer. So let us momentarily rewrite our equations in the form 
\begin{equation}\label{equ:rewrite1} X_i^{Y_1^{C'(i,j)} \cdots Y_n^{C'(i,j)}} = X_j^{Y_1^{C''(i,j)} \cdots Y_n^{C''(i,j)}}, 
\end{equation} where $C'(i,j),C''(i,j)$ are non-negative integers, for all $(i,j) \in R$. Now, it is only possible for one side of such an equation to be composite if the equation is in the form $X_i = X_j$, $i\not=j \in [n]$, as we are always assuming that $Y_1,\ldots,Y_n, X_1,\ldots,X_n >1$. But we have assumed that we have no such equation in our system. 
\paragraph{}
With this technicality aside, let us rewrite each equation of (\ref{equ:rewrite1}) by applying the function $\nu^2 = \nu \circ \nu$ to both sizes of every equation in $\mathcal{E}$ to obtain
\[ C_1(e)\nu(Y_1) + \cdots + C_n(e)\nu(Y_n) = \nu^2(X_j) - \nu^2(X_i),
\] after rearranging, for each $e = (i,j) \in R$. Now, given a cycle $C$ of $R$, we may eliminate the the $\nu^2$ terms by summing over the cycle $C$, multiplying by $\pm 1$ according to the orientation of each $e \in C$. We obtain the equations 
\[ \sum_{e \in C} (-1)^{d(e)}\left(C_1(e)\nu(Y_1) + \ldots + C_n(e)\nu(Y_n) \right) = 0,
\] for each cycle $C$ in $R$. Now, if $y_1,\ldots,y_n$ form a monochromatic solution of the above equation, we have that $f(y_1) = \cdots = f(y_n)$
and therefore $c(\nu(y_1)) = \cdots = c(\nu(y_n))$. So if we put $u_1 = \nu(y_1), \ldots, u_n = \nu(y_n)$, we see that $u_1,\ldots,u_n$ satisfy the equation $\mathcal{L}$ and $c(u_1)  = \cdots = c(u_n)$. However, our choice of $c$ forbids this situation. Therefore $\mathcal{E}$ is not partition regular. This completes the proof of Theorem~\ref{thm:MainClassification}. \qed

\section{Acknowledgements}
I should like to thank B\'{e}la Bollob\'{a}s and Imre Leader for comments.

\Addresses

\begin{thebibliography}{9}

\bibitem{Azcel} A. Aczel, \emph{Fermat's Last Theorem: Unlocking the secret of an ancient mathematical problem}, Four Walls Eight Windows, New York (1996).

\bibitem{BHLS} B. Barber, N. Hindman, I. Leader, D. Strauss, \emph{Partition regularity without the columns property}, Proc. Amer. Math. Soc. {\bf 143} (2015), 3387-3399.


\bibitem{BL} V. Bergelson, A. Leibman, \emph{Polynomial extensions of van der Waerden and Szemer\'{e}di's theorems}, J. Amer. Math Soc. {\bf 9} (1996), 725-753.
\bibitem{BBmgt} B. Bollob{\'a}s, \textit{Modern Graph Theory}, Graduate Texts in Mathematics, vol. 184, Springer-Verlag, New York, 1998.

\bibitem{Br} T. Brown, \emph{Monochromatic solutions of exponential equations}, Integers, {\bf 15 A} (2015).

\bibitem{PeterAndSarkozy} P. Csikv\'{a}ri, K. Gyarmati, A. S\'{a}rk\"{o}zy, \emph{Density and Ramsey type results on algebraic equations with restriced solution sets}, Combinaorica, (2012) 32(4): 425-449.  

\bibitem{Edwards} H. Edwards, \emph{Fermat's Last Theorem: a Genetic Introduction to Algebraic Number Theory}. Graduate Texts in Mathematics 50. New York: Springer-Verlag (1997).


\bibitem{ErdosGraham} P. Erd\H{o}s, R.L. Graham, \emph{Old and new problems and results in combinatorial number theory},
L'Enseignement Math\'{e}matique, Geneva, (1980).
\bibitem{GRS} R.L. Graham, Bruce L. Rothchild, J.H. Spenser, \emph{Ramsey Theory}. Wiley-Interscience Series in Discrete Mathematics. New York: John Wiley \& Sons (1980).

\bibitem{HalesJewett} A. W. Hales, R. I. Jewett, \emph{Regularity and positional games}, Trans. Amer. Math. Soc. {\bf 106} (1963), 222-229.
\bibitem{Hind} N. Hindman, \emph{Finite sums from sequences within cells of a partition of N}, J. Combin. Theory Ser. A, {\bf 45} (1987), 300-302.
\bibitem{HindSurvey} N. Hindman, \emph{Partition regularity of matrices}, Integers, {\bf 7} (2007).
\bibitem{HindSumsAndProds} N. Hindman, \emph{Partitions and sums and products of integers}, Trans. Amer. Math. Soc.
247 (1979), 227-245.
\bibitem{ImagePartReg} N. Hindman, I. Leader, \emph{Image partition regularity of matrices}, Comb. Prob. and Comp.
{\bf 2} (1993), 437-463.
\bibitem{NonConstSol} N. Hindman, I. Leader, \emph{Nonconstant monochromatic solutions to systems of linear equations}, Topics in Discrete Math. Springer, Berlin, (2006), 145-154.
\bibitem{HLS0} N. Hindman, I. Leader, D. Strauss, \emph{Image partition regular matrices - solutions in central sets}, Trans. Amer. Math. Soc. {\bf 355} (2003), 1213-1235.
\bibitem{HLS} N. Hindman, I. Leader, D. Strauss, \emph{Open problems in partition regularity}, Combinatorics, Probability and Computing, {\bf 12} (2003), 571 - 583.
\bibitem{Mill} K. Milliken, \emph{Ramsey's theorem with sums or unions}, J. Combin, Theory (Series A) {\bf 18} (1975), 276-290.
\bibitem{Moreira} J. Moreira, \emph{Monochromatic sums and products in} $\mathbb{N}$, https://arxiv.org/abs/1605.01469 (2016).
\bibitem{Ra0} R. Rado, \emph{Verallgemeinerung Eines Satzes von van der Waerden mit Anwendugen auf en ein Problem der Zahlentheorie}, Sonderausg. Sitzungsber. Preuss. Akad. Wiss. Phys.-Math. Klass {\bf 17} (1933), 1-10.
\bibitem{Ra} R. Rado, \emph{Studien zur Kombinatorik}, Math. Zeit. {\bf 36} (1933), 242-280.
\bibitem{Ra'} R. Rado, \emph{Some partition theorems}, Colloquia Mathematica Societatis Janos Bolyai 4. Combinatorial Theory and Its Applications, Balatonfured, Hungary, North Holland (1969).
\bibitem{ExpPatterns} J. Sahasrabudhe, \emph{Exponential Patterns in Arithmetic Ramsey Theory}, https://arxiv.org/ abs/1607.08396 (2016).
\bibitem{Sanders} J. Sanders, \emph{A Generalization of Schur's Theorem}, Dissertation, Yale University (1969).
\bibitem{Sch} I. Schur, \emph{Uber die Kongruenz} $x^m + y^m \equiv z^m \mod p $, Jahresber Deutsch. Math. Verein. {\bf 25} (1916), 114-117.
\bibitem{Si} A. Sisto \emph{Exponential Triples}, Electron. J. Combin. , {\bf 18} (2011), Paper 147.
\bibitem{Sz} E. Szemer\'{e}di, \emph{On sets of integers containing no $k$ elements in arithmetic progression}, Acta Arith., (1975) 27:299-345. 
\bibitem{TaoVu} T.C. Tao, V.H. Vu, \emph{Additive Combinatorics}. volume 105 of Cambridge Studies in Advanced Mathematics. Cambridge University Press, Cambridge, (2006).
\bibitem{Tay} A. Taylor, \emph{A canonical partition relation for finite subset of} $\omega$, J. Combin. Theory Ser. A, {\bf 21} (1976), 137-146.
\bibitem{vdW} B. van der Waerden, \emph{Beweis einer Baudet’schen Vermutung}, Nieuw Arch. Wiskunde {\bf 19} (1927),
212-216.
\bibitem{Mark} M. Walters, \emph{Combinatorial proofs of the polynomial van der Waerden theorem and the polynomial Hales-Jewett theorem}.  J. London Math. Soc. {\bf 61} (2000), 1–12.

\end{thebibliography}
\end{document}